\documentclass[12pt, a4paper]{amsart}
\usepackage{a4, a4wide}
\usepackage{amssymb}
\usepackage{amsmath}
\usepackage{amsthm}
\usepackage{amstext}
\usepackage{amscd}
\usepackage{latexsym}
\usepackage{graphics}
\usepackage{color}
\usepackage{enumitem}
\usepackage[all]{xy}
\usepackage{tikz-cd}

\usepackage[colorlinks, pagebackref]{hyperref}
\hypersetup{
  colorlinks=true,
  citecolor=blue,
  linkcolor=blue,
  urlcolor=blue}
\makeatletter
\def\@tocline#1#2#3#4#5#6#7{\relax
  \ifnum #1>\c@tocdepth % then omit
  \else
    \par \addpenalty\@secpenalty\addvspace{#2}%
    \begingroup \hyphenpenalty\@M
    \@ifempty{#4}{%
      \@tempdima\csname r@tocindent\number#1\endcsname\relax
    }{%
      \@tempdima#4\relax
    }%
    \parindent\z@ \leftskip#3\relax \advance\leftskip\@tempdima\relax
    \rightskip\@pnumwidth plus4em \parfillskip-\@pnumwidth
    #5\leavevmode\hskip-\@tempdima
      \ifcase #1
      \or\or \hskip 2em \or \hskip 2homologyem \else \hskip 3em \fi%
      #6\nobreak\relax
    \dotfill\hbox to\@pnumwidth{\@tocpagenum{#7}}\par
    \nobreak
    \endgroup
  \fi}
\makeatother

\theoremstyle{plain}

\newtheorem{theorem}{Theorem}[section]
\newtheorem{question}[theorem]{Question}
\newtheorem{lemma}[theorem]{Lemma}
\newtheorem{corollary}[theorem]{Corollary}
\newtheorem{proposition}[theorem]{Proposition}

\theoremstyle{definition}

\newtheorem{remark}[theorem]{Remark}

\newtheorem{definition}[theorem]{Definition}
\newtheorem{example}[theorem]{Example}

\numberwithin{equation}{section}
\newcommand{\inj}{\hookrightarrow}
\newcommand{\tensor}{\otimes}

\newcommand{\CH}{{\rm CH}}

\newcommand{\Proj}{{\bf{Proj}}}

\newcommand{\Pic}{{\rm Pic}}

\newcommand{\Spec}{{\rm Spec \,}}

\newcommand{\Sym}{{\rm Sym}}

\newcommand{\Z}{{\mathbb Z}}
\newcommand{\A}{{\mathbb A}}
\renewcommand{\P}{{\mathbb P}}

\def\<{\langle}
\def\>{\rangle} 
\def\-{\overline} 
\def\~{\widetilde}
\def\^{\widehat}

\def\@{\mathcal}
\def\!{\mathscr}
\def\#{\mathbb}
\def\_{\underline}

\input{xy}
\xyoption{all}
\begin{document}
\title{$\A^1$-connectedness of moduli of vector bundles on a curve}

\author{Amit Hogadi}
\address{Department of Mathematical Sciences, Indian Institute of Science Education and Research Pune, Dr. Homi Bhabha Road, Pashan, Pune 411008, India.}
\email{amit@iiserpune.ac.in}

\author{Suraj Yadav}
\address{Universit\"{a}t Regensburg, Universit\"{a}tsstr. 31, 93040, Regensburg, Germany}
\email{suraj.yadav@ur.de}

%\subjclass[2010]{14C15, 14C25, 19E15 (Primary)}
%\keywords{algebraic cycles; motivic cohomology; rost nilpotence}

\begin{abstract} 
In this note we prove that the moduli stack of vector bundles on a curve with a fixed determinant is $\A^1$-connected. We obtain this result by classifying vector bundles on a curve up to $\A^1$-concordance.Consequently we classify $\P^n$-bundles on a curve up to $\A^1$-weak equivalence, extending a result in \cite{Asok-Morel} of Asok-Morel. We also give an explicit example of a variety which is $\A^1$-$h$-cobordant to a projective bundle over $\P^2$ but does not have the structure of a projective bundle over $\P^2$, thus answering a question of Asok-Kebekus-Wendt \cite{AKW}. 
\end{abstract}

\maketitle
%\tableofcontents

\setlength{\parskip}{2pt plus1pt minus1pt}

\section{Introduction}
%\newpage
Let $C$ be a smooth projective curve of genus $g $ over a field $k$. Fix a line bundle $\@{L} \in \mathrm{Pic}(C)$. Consider the following moduli stack $Bun_{n,\@L}$,
\begin{center}
	$ Bun_{n,\@L} (Y)$ = \{ category of rank $n $ vector bundles on $C \times Y$, such that for any object $\@V$, $\mathrm{det}\@V \cong  p^\ast(\@L),$ where $p: C \times Y \rightarrow C \} $
\end{center}
This is a smooth algebraic stack \cite[Prop 1.3]{alp}. Any stack can be regarded as a simplicial sheaf via the nerve construction (see \cite{hollander}) and thus  it defines an object in the $\A^1$-homotopy category. Therefore it makes sense to talk about $\A^1$-connectedness of $Bun_{n,\@L}$. Following is the main theorem of this note:

\begin{theorem}\label{main}
	$Bun_{n,\@L}$ is $\A^1$-connected for any curve $C$ over an infinite field $k$ and $\@L \in \Pic(C)$.
\end{theorem}

The proof of the Theorem \ref{main} relies on finding an explicit $\A^1$-concordance (see \cite[Definition 5.1]{AKW} or Definition \ref{con}) between a vector bundle $\@E$ of rank $n$ and determinant $\@L$ to the vector bundle $\@O_C^{n-1} \oplus \@L$. This is achieved by induction on $n$.  In the course of this proof, we also achieve the classification of vector bundles of a given rank on the curve $C$ up to $\A^1$-concordance (Theorem \ref{conco}). Once the question of $\A^1$-connectedness of $ Bun_{n,\@L} $ is settled, it's natural to wonder the same  about its open substack $ Bun_{n,\@L} ^s$, the moduli of stable vector bundles for a curve of genus greater than 1. Assuming that $n$ and degree of $\@L$ are co-prime, $k$ algebraically closed, the coarse moduli space is known to be rational (\cite[Theorem 1.2]{rat}) and one may be tempted to conclude that $\A^1$-connectedness of an algebraic stack is dictated by that of its coarse moduli space (if it exists). However in Example \ref{stack} we show that this simply is not the case. We give an example of ``stacky" $\P^1$-an orbifold with $\P^1$ as a  coarse moduli space-which is not $\A^1$-connected. It should also be noted that the $\A^1$-connectedness is not preserved under rationality of a morphism of schemes as illustrated by $\mathbb{G}_m \hookrightarrow \A^1$. \\

Related to $\A^1$-concordance is the notion of an $\A^1$-$h$-cobordism (\cite[Definition 3.1.1]{Asok-Morel}  or Definition \ref{cob}). Furthermore projectivizations of two $\A^1$-concordant vector bundles are $\A^1$-$h$-cobordant. As an application of $\A^1$-connectedness of $ Bun_{n,\@L} $, we obtain the following theorem which classifies $\P^n$-bundles over any curve of genus $g$ up to $\A^1$-weak equivalence. This extends the result on classification of $\P^n$-bundles over $\P^1$ given in \cite{Asok-Morel}.\\
\begin{theorem}\label{curve}
	Let $X=\P_{C}(\@E)$ and $Y=\P_{C}(\@F)$ be $\P^n$-bundles over $C$, where $C$ lies over an infinite field. Then the following are equivalent : 
	\begin{enumerate}
	\item $X$ and $Y$ are $\A^1$- weakly equivalent.
	\item $X$ and $Y$ are $\A^1$-h-cobordant. 
	\item ${\rm det}(\@E) \tensor {\rm det}(\@F)^{-1} = \@L^{\otimes n+1}$, for some $\@L \in \Pic (C)$.
	\end{enumerate}
\end{theorem}

In another application of our theorem we answer a question raised in \cite{AKW}: whether a variety which is $\A^1$-$h$-cobordant to a $\P^1$-bundle over $\P^2$ has a structure of $\P^1$-bundle over $\P^2$. The answer is no and we prove in the following theorem that the suggested example in op. cit. indeed works.\\

\begin{theorem} \label{count}  Let $ X:=\P_{\P^1}(\@E)$, where $ \@E := \@O \oplus \@O(-1) \oplus \@O(1)$ on $\P^1_k$ (where $k$ is an infinite field). Then $X$ is $\A^1$-$h$-cobordant to $\P^1_k \times \P^2_k$ but doesn't have the structure of a $\P^1_k$-bundle over $\P^2_k$.
\end{theorem}

\section{Classification of vector bundles on a curve up to $\A^1$-concordance}
In this section we classify vector bundles on a curve up to $\A^1$-concordance (Theorem \ref{conco}) and obtain the proof of Theorem \ref{main} as a consequence of that. Recall the following definition from \cite{AKW}.
\begin{definition}\cite[Definition 5.1]{AKW} \label{con}
Let $X$ be a scheme over a field $k$. Then two given vector bundles $\@E_0$ and $\@E_1$ on $X$, are said to be directly $\A^1$-concordant if there exists a vector bundle $\@E$ on $X \times \A^1$ such that $i_0^{\ast}\@E \cong \@E_0$ and $i_1^{\ast}\@E  \cong\@E_1$, where $i_k: X \times\{k\} \inj X \times \A^1$, for $k=0,1$. $\@E_0$ and $\@E_1$ on $X$ are $\A^1$-concordant if they are equivalent under the equivalence relation generated by direct $\A^1$-concordance.
\end{definition}
%\begin{remark}
%Let $X$ be a noetherian integral separated scheme with is regular in codimesnion 1.	If $\@E_0$ and $\@E_1$ are $\A^1$-concordant then $\mathrm{det}(\@E_0) \cong \mathrm{det}(\@E_1)$. To prove this it's enough to assume $\@E_0$ and $\@E_1$ are directly $\A^1$-concordant, such that $i_0^{\ast}\@E = \@E_0$ and $i_1^{\ast}\@E =\@E_1$, for some vector bundle $\@E$ on $X \times \A^1$. Then by property of chern classes we have $(i_{k})_{\ast} (c_1(i_{k}^{\ast}\@E)) = c_1(\@E)$, for $k=0,1$. But $c_1(i_{k}^{\ast}\@E) = c_1(\mathrm{det}\@E_k)$ and $(i_{0})_{\ast} = (i_{1})_{\ast}$, as $\Pic(X) \cong \Pic(X \times \A^1)$. Therefore $\mathrm{det}(\@E_0) \cong \mathrm{det}(\@E_1)$.
%\end{remark}
\begin{lemma}\label{sum}
Let $\@E_0$ and $\@E_1$ be $\A^1$-concordant vector bundles on a normal variety $X$ and $\@V$ be a vector bundle on $X\times \A^1$. Then $(i_0^{\ast}(\@V) \otimes \@L) \oplus \@E_0$ and $(i_1^{\ast}(\@V) \otimes \@L) \oplus \@E_1$  are  $\A^1$-concordant, for any $\@L \in \Pic (X)$.
\end{lemma}
\begin{proof}
It is enough to prove the lemma in the case when $\@E_0$ and $\@E_1$ are directly $\A^1$-concordant. Let the direct $\A^1$-concordance be given by a vector bundle $\@E$ on $X \times \A^1$. Note that $p^{\ast} : \Pic (X) \rightarrow \Pic (X \times \A^1)$ (where $p: X \times \A^1 \rightarrow X$) is an isomorphism (see \cite[II, Prop. 6.6]{hart}) with the inverse given by $i_0^{\ast} = i_1^{\ast}$. Then the lemma immediately follows from the definition by considering the vector bundle $(\@V \otimes p^{\ast}\@L) \oplus \@E$ and the fact that the pullback functor commutes with the direct sums. 
\end{proof}
In light of the previous lemma, the following corollary is rather obvious but we state it nevertheless, keeping in mind its direct application in the proof of Theorem \ref{main}.
\begin{corollary}\label{obv}
Let $\@E_0$ and $\@E_1$ are $\A^1$-concordant vector bundles on a normal variety $X$ Then the following statements hold
\begin{enumerate}
	\item $\@O_X^n \oplus \@E_0$ and $\@O_X^n \oplus \@E_1$ are $\A^1$-concordant for any $n \geq 0$.
	\item $\@O_X(m) \oplus \@E_0$ and $\@O_X(m) \oplus \@E_1$ are $\A^1$-concordant for any $m$.
\end{enumerate}
\end{corollary}
\begin{proof}
	For the first statement take $\@V = \@O_{X \times \A^1}^n$, keeping the notation of the previous lemma in mind.\\
	For the second statement take $\@V = \@O_{X \times \A^1}$ and $\@L = p^{\ast} \@O_X(m)$.
\end{proof}

%\begin{remark}\label{convect}
%By the definition, any two $F$ valued points of $Bun_{n,\@L}$ are two rank $n$ (with determinant condition) vector bundles, say $\@E_0$ and $\@E_1$ on $C$ and a morphism $\A^1_F \rightarrow Bun_{n,\@L}$ is a vector bundle $\@E$ on $C \times \A^1$. Then $\@E_0$ and $\@E_1$ are naively $\A^1$-homotopic if and only if they are $\A^1$-concordant.
%\end{remark}

%\subsection{Stacks as Simplicial sheaves}
%Any category (in particular any groupoid), via nerve construction, can be considered as a Simplicial set. Therefore any functor, fibered in groupids $F: Sm/k \rightarrow Grpd$ -where $Grpd$ is the category of Groupoids-  can be regarded as a simplicial presheaf. As any algebraic stack is a functor fibered in groupoids, (along with some extra glueing requirements which we don't need here), one can talk about $\A^1$ homotopy groups of a stack.  For more details see Section 2.2 of \cite{Chou}.

\subsection{$\A^1$-concordance via Ext classes}
Now we look at a way of constructing $\A^1$-concordance between vector bundles.
% Let $X$ be a scheme over $k$ and suppose there is a non trivial exact sequence of vector bundles on $X$
 %\begin{center}
 %	$0 \rightarrow \@E_0 \rightarrow \@E \rightarrow \@E_1 \rightarrow 0$
 %\end{center}
%Therefore $\@E$ is a non trivial element in $\rm Ext^1(\@E_1, \@E_0)$.\\
%Now consider the the moduli functor $\boldsymbol{\rm Ext^1}(\@E_1, \@E_0)$ given by $Y \mapsto \rm Ext^1(p^{\ast}\@E_1, p^{\ast}\@E_0)$, where $p: X \times Y \rightarrow X$.  By \cite{lange}, Proposition 3.1, this functor is representable by $\A^n$, where $n= \mathrm{dim}(\rm Ext^1(\@E_1, \@E_0))$ as a vector space over $k$. Note that $n>0$ by existence of the non trivial short exact sequence above. Therefore by representability there is a universal class $\@V$ (of vector bundle) on $X \times \A^n$ whose pullback to $X\times{t_i} \inj X \times \A^n$, $i=0,1$ is $\@E$ and $\@E_0 \oplus \@E_1$ respectively for some $t_i \in \A^n$. We connect these $t_i$'s via an $\A^1$ and restrict $\@V$ to $\A^1$ to obtain a direct $\A^1$ concordance between $\@E$ and $\@E_0 \oplus \@E_1$. The preceding discussion can be summarised in following proposition as 
\begin{proposition}\label{bridge}
Let  	$0 \rightarrow \@E_0 \rightarrow \@E \rightarrow \@E_1 \rightarrow 0$ be any short exact sequence of vector bundles on a projective scheme $X$ over a field $k$. Then $\@E$ is directly $\A^1$-concordant to $\@E_0 \oplus \@E_1$.
\end{proposition}
\begin{proof}
Consider $\@E$ as an element in $\rm Ext^1(\@E_1, \@E_0)$. If $\@E$ is trivial then our claim is obvious, so assume to the contrary. Consider the moduli functor $\boldsymbol{\rm Ext^1}(\@E_1, \@E_0)$ given by $Y \mapsto {\rm Ext}^1(p^{\ast}\@E_1, p^{\ast}\@E_0)$, where $p: X \times Y \rightarrow X$.  It's well known (\cite[Proposition 3.1]{lange}) that this functor is representable by $\A^n_k$, where $n= \mathrm{dim}(\rm Ext^1(\@E_1, \@E_0))$ as a vector space over $k$ and $n>0$ by the assumption that $\@E$ is non trivial. Therefore by representability there is a universal class $\@V$ (of vector bundle) on $X \times \A^n_k$ whose pullback to $X\times{t_i} \inj X \times \A^n_k$, $i=0,1$ is $\@E$ and $\@E_0 \oplus \@E_1$ respectively for some $k$-rational points $t_i \in \A^n_k$. For any two given $k$-rational points in $\A^n_k$ (in our case $t_0$ and $t_1$) there is a closed embedding $i:\A^1_k \hookrightarrow \A^n_k$ such that the composition $\Spec k \xrightarrow{0} \A^1_k \xrightarrow{i}\A^n_k$ is $t_0$ and the composition $\Spec k \xrightarrow{1} \A^1_k\xrightarrow{i} \A^n_k$ is $t_1$. Now consider the pullback of the universal class, $(id_X \times i)^{\ast}\@V$, via the map $id_X \times i : X \times \A^1_k \to X \times \A^n_k$. By construction the vector bundle $(id_X \times i)^{\ast}\@V$ on $X \times \A^1_k$ gives a direct $\A^1$-concordance between $\@E$ and $\@E_0 \oplus \@E_1$
\end{proof}
\subsection{Classification result and proof of Theorem \ref{main}}
\begin{theorem}\label{conco}
	Let $\@E$ and $\@F$ be rank $n$ vector bundles on the curve $C$. Then the following hold 
	\begin{enumerate} 
	\item $\@E$ is $\A^1$-concordant to $\@O_C^{n-1} \oplus \det(\@E)$.
	\item $\@E$ is $\A^1$-concordant to $\@F$ iff $\det(\@E) \cong \det(\@F)$. 
	\end{enumerate}
\end{theorem}
\begin{proof}  We first prove (1) for the case when $n = 2$. For the general case we will use induction.\\ \textbf{Case 1:} $n=2$. First assume $\@E$ is globally generated and denote $\det(\@E)$ by $\@L$. Then by \cite[II, Exercise 8.2]{hart},  we have the following short exact sequence
	\begin{equation}\label{2.1}
		0 \rightarrow \@O_C \rightarrow \@E \rightarrow \@E' \rightarrow 0
	\end{equation}
	where $\@E'$ is a line bundle. By the Whitney sum formula of Chern classes (\cite[Theorem 5.3(c)]{EH}) $$c_1(\@L) = c_1(\@E) = c_1(\@O_C) + c_1(\@E') = c_1(\@E').$$
	Therefore by Proposition \ref{bridge}, $\@E$ is directly $\A^1$-concordant to $\@O_C \oplus \@L$. 
	For a general $\@E$, choose $m >>0$ such that $\@E(m)$, $\@L(m)$ are globally generated. Then again by applying \cite[II, Exercise 8.2]{hart} we get a short exact sequence for $\@E(m)$ which we tensor by $\@O(-m) $ to obtain the following short exact sequence.s
	\begin{equation}
		0 \rightarrow \@O_C(-m) \rightarrow \@E \rightarrow \@L(m) \rightarrow 0
	\end{equation}
	This proves $\@E$ is directly $\A^1$-concordant to $\@O_C(-m) \oplus \@L(m)$. As the final step we now prove that $\@O_C(-m) \oplus \@L(m)$ is directly $\A^1$-concordant to $\@O_C \oplus \@L$.
	Note that $m$ is chosen such that $\@L(m)$ is globally generated, therefore $\@O_C(m) \oplus \@L(m)$ is globally generated. Hence we have a short exact sequence which shows  $\@O_C(-m) \oplus \@L(m)$ is directly $\A^1$-concordant to $\@O_C \oplus \@L$.
		\begin{equation}\label{2.3}
0 \to \@O_C(-m) \to \@O_C \oplus \@L \to \@L(m)\to 0
	\end{equation}
Therefore, $\@E$ is $\A^1$-concordant to $\@O_C\oplus \@L$.\\
\textbf{Case 2:} Now we handle the general case. So assume $n >2$ and choose $m$ such that $\@E(m)$ and $\@L(m)$ are globally generated. Then we have a short exact sequence giving a direct $\A^1$-concordance between $\@E$ and $\@O_C(-m) \oplus \@E'$, where $\@E'$ is a vector bundle of rank $n-1$ with determinant $\@L(m)$. By induction, $\@E'$ is $\A^1$-concordant to $\@O_C^{n-2} \oplus \@L(m)$. Therefore by the second statement of Corollary \ref{obv}  we have an $\A^1$-concordance between $\@O_C(-m)\oplus\@O_C^{n-2} \oplus \@L(m)$ and $\@O_C(-m) \oplus \@E'$. Hence  $\@E$ is $\A^1$-concordant to $\@O(-m)\oplus\@O_C^{n-2} \oplus \@L(m)$. Now by the  short exact sequence \ref{2.3}, $\@O_C\oplus \@L$ is directly $\A^1$-concordant to $\@O_C(-m) \oplus \@L(m)$, which implies --  by first statement of  Corollary \ref{obv} -- that $\@O(-m)\oplus\@O_C^{n-2} \oplus \@L(m)$ is directly $\A^1$-concordant to $\@O_C^{n-1} \oplus \@L$, thus finishing the proof of (1).\\
For proving (2), we first observe that if $\det(\@E) \cong \det(\@F)$ then (1) implies that $\@E$ is $\A^1$-concordant to $\@F$. Hence it remains to show that if $\@E$ is directly $\A^1$-concordant to $\@F$ then $\mathrm{det}(\@E)\cong \mathrm{det}(\@F)$.\\ So assume that $\@E$ is directly $\A^1$-concordant to $\@F$, which by definition gives us a vector bundle $\@E'$ on $C \times \A^1$ such that $i_0^{\ast}\@E' \cong \@E$ and $i_1^{\ast}\@E' \cong \@F$. We have $ c_1(\mathrm{det}\@E) = c_1(i_{0}^{\ast}\@E') = i_0^{\ast}(c_1(\@E'))$, where the first equality follows from the isomorphism $i_0^{\ast}\@E' \cong \@E$ and the fact that for any vector bundle $\@V$,  $c_1(\@V)=c_1(\mathrm{det}\@V)$, while the second equality is the functoriality of Chern classes (\cite[Theorem 5.3(d)]{EH}). Similarly we have $ c_1(\mathrm{det}\@F) = c_1(i_{1}^{\ast}\@E') = i_1^{\ast}(c_1(\@E'))$. Moreover, $(i_{0})^{\ast} = (i_{1})^{\ast}: \CH^1(C \times \A^1) \to \CH^1(C) $, which implies $c_1(\mathrm{det}\@E)=c_1(\mathrm{det}\@F)$  Therefore $\mathrm{det}(\@E) \cong \mathrm{det}(\@F)$.
	
\end{proof}

\noindent Before we proceed with the proof of Theorem \ref{main}, we recall some standard definitions.\\
\begin{definition}
	Let $\@X$ be a simplicial sheaf and $U$ a scheme. Then $x$ and $y$  in $\@X(U)$ are said to be naively $\A^1$-homotopic if there exists $f: \A^1_U \rightarrow \@X$ such that $f_0 = x$ and $f_1 =y$, where $f_i $ is the composition $ U \xrightarrow{i} \A^1_U \xrightarrow{f} \@X$, for $i=0,1$.
\end{definition}	

\begin{definition}
For a given simplicial sheaf $\@X$ we define $\@S(\@X)$ to be the Nisnevich sheafification of the presheaf  $U \mapsto \@X(U)/\sim$, where $\sim$ is the equivalence relation generated by naive $\A^1$-homotopies.
\end{definition}
\noindent The following standard lemma will be required in our proof. It's essentially \cite[Section 2, Corollary 3.22]{Morel-Voevodsky} combined with \cite[Lemma 6.1.3.]{Morel-connectivity}, 
\begin{lemma}\label{weak}
 A simplicial sheaf $\@X$ is $\A^1$-connected if $\@S(\@X)(F) = \ast$ for every finitely generated field extension $F$ over $k$.
\end{lemma}

Now we have all the ingredients in place to prove Theorem \ref{main}.

\begin{proof}[ Proof of Theorem \ref{main}] 
	We regard $Bun_{n,\@L}$ as a simplicial sheaf. By definition, any two $F$-valued points of $Bun_{n,\@L}$ are two rank $n$ (with determinant condition) vector bundles, say $\@E_0$ and $\@E_1$ on $C$. A morphism $\A^1_F \rightarrow Bun_{n,\@L}$ is a vector bundle $\@E$ on $C \times \A^1$. Then $\@E_0$ and $\@E_1$ are naively $\A^1$-homotopic if and only if they are $\A^1$-concordant. By Theorem \ref{conco} both $\@E_0$ and $\@E_1$ are $\A^1$-concordant to $\@O_C^{n-1} \oplus \@L$. Hence they are $\A^1$-concordant to each other. Therefore by Lemma \ref{weak}, $Bun_{n,\@L}$ is $\A^1$-connected.
\end{proof}

Motivated by the question of $\A^1$-connectedness of moduli stack of stable vector bundles, we observe in the example below that there does not seem to be an immediate way of concluding $\A^1$-connectedness of a stack by looking at its coarse moduli space. \begin{example}\label{stack}
	Let $C$ be a curve of genus $2$ over a field with characteristic not equal to $2$. In particular it is a hyperelliptic curve (See \cite[IV, Exercise 1.7(a)]{hart}). Therefore, there is a finite morphism $f: C \rightarrow \P^1$ of degree $2$ and we have an action of the finite group $G:=\Z/2\Z$ on $C$. By \cite[IV, Exercise 2.2(a)]{hart}, such a morphism is unramified at all but 6 points (denoted as closed subscheme $Z'$) of $C$. So the action of $G$ is free on $C \setminus Z'$. Let $Z$ denote the closed subset in $\P^1$ corresponding to the 6 branched points. The quotient stack $\big[C/G\big]$ has coarse moduli space $\P^1$ and the morphism $\pi: \big[C/G\big] \rightarrow \P^1$ gives an isomorphism of an open subscheme of $\big[C/G\big] $ with $\P^1 \setminus Z$. See \cite[Example 8.1.12]{ols} for more details on Quotient stacks.\\
Let $E(G)$ denote the simplicially contractible, simplicial sheaf with $E(G)_n = G^{n+1}$(See \cite[Example 1.11, page 128]{Morel-Voevodsky}). The morphism $C \to \big[C/G\big] $ is a $G$-torsor. 
 Moreover $G$ acts freely on the space $E(G) \times C$ and $(E(G) \times C)/G \simeq \big[C/G\big] $.  $\pi_0^{\A^1}(G) \cong \Z/2\Z$ being a finite abelian group is a strictly $\A^1$-invariant sheaf. So all the hypothesis of the statement of \cite[Theorem 6.50]{Morel} are satisfied, as a consequence of which we obtain the following long exact sequence of $\A^1$-homotopy groups/pointed sets, where $\ast$ is a chosen basepoint.
 $$ \cdots \to \pi_0^{\A^1}(G, \ast) \to \pi_0^{\A^1}(E(G) \times C, \ast)  \to \pi_0^{\A^1}( \big[C/G\big], \ast) \to \ast$$ But on the account of $E(G)$ being simplicially contractible  and $C$ being $\A^1$-rigid (as all curves of genus $g > 0$ are) $\pi_0^{\A^1}(E(G) \times C) \cong \pi_0^{\A^1}(C) \cong C $. So by long exact sequence, $\big[C/G\big]$ being $\A^1$-connected would imply surjection of finite group $\Z/2\Z$ on $C \cong \pi_0^{\A^1}(C)$, which can not happen.
\end{example}
\begin{remark}\label{2.10}
	 By definition, any hyperelliptic curve of genus $g$ admits a finite map of degree $2$ to $\P^1$. By Hurwitz's theorem, such a morphism has $2g+2$ ramified points. Therefore Example \ref{stack} can be generalized to an hyperelliptic curve of any genus (which is necessarily greater than 1).
\end{remark}
\section{Applications}
As applications of the results in the previous section we give a proof of Theorem \ref{curve} and Theorem \ref{count}. We first recall the following definition from \cite{Asok-Morel}. 

\begin{definition}\cite[Definition 3.1.1]{Asok-Morel}\label{cob}
	Let $X_0$ and $X_1 $ be smooth and proper varieties over $k$.  They are directly $\A^1$-$h$- cobordant if there exists a smooth scheme $X$ with $f: X \rightarrow \A^1$  a proper surjective morphism such that 
	\begin{enumerate}
		\item the fibers of $f$ over $0$ and $1$ are $X_0$ and $X_1$ respectively
		\item the natural maps $X_i \inj X$ for $i=0,1$ are $\A^1$-weak equivalences.
		
	\end{enumerate}
	$X_0$ and $X_1 $  are $\A^1$-$h$-cobordant if they are equivalent under the equivalence relation generated by direct $\A^1$-$h$-cobordance.
\end{definition}
While $\A^1$-concordance is a relation between vector bundles, $\A^1$-$h$-cobordism a relation between proper schemes. Note that by \cite[Lemma 6.4]{AKW}, projectivizations of $\A^1$-concordant vector bundles are $\A^1$-$h$-cobordant.\\
Recall that, given a locally free sheaf $\@E$ of rank $n+1$ on a scheme $X$, the associated $\P^n$-bundle, denoted $\P_X(\@E)$ is the scheme $\Proj_X(\Sym(\@E))$. Here $\Proj$ is the relative proj construction and $\Sym(\@E)$ is the symmetric algebra of $\@E$ as an $\@O_X$-module. See \cite[II, page 162]{hart} for more details.\\
%\begin{proposition} \label{concob} \fixme{Dont state the lemma just use it directly}
%	Let $X$ be a smooth projective variety and $\@E_0$, $\@E_1$ be directly $\A^1$-concordant vector bundles on $X$, given by vector bundle $\@E$ on $X \times \A^1$. Then their projectivisations $\P_X(\@E_0)$ and $\P_X(\@E_1)$ are directly $h$-cobordant via the map $\P_{X\times \A^1}(\@E) \rightarrow X \times \A^1 \rightarrow \A^1$.
%\end{proposition}
We now paraphrase the classification of $\P^n$-bundles on $\P^1$ up to $\A^1$-weak equivalence proved in \cite{Asok-Morel} to highlight that Theorem \ref{curve} is its direct generalization to an arbitrary smooth projective curve.
\begin{proposition}\cite[Proposition 3.2.10]{Asok-Morel}\label{3.2}
Let $X := \P (\@O_{\P^1}^{n} \oplus \@O_{\P^1}(a))$ and $Y := \P (\@O_{\P^1}^{n} \oplus \@O_{\P^1}(b))$ be two $\P^n$-bundles over $\P^1$. Then the following statements are equivalent:
\begin{enumerate}
	\item $X$ and $Y$ are $\A^1$- weakly equivalent.
	\item $X$ and $Y$ are $\A^1$-h-cobordant.
	\item $n+1$ divides $a-b$.
\end{enumerate}  
\end{proposition}

Note that in case of $ C =\P^1$ the condition ${\rm det}(\@E) \tensor {\rm det}(\@F)^{-1} = \@L^{\otimes n+1}$ in Theorem \ref{curve} exactly translates to the fact that $n+1$ divides $a-b$ as stated in the  Proposition \ref{3.2}. This is due to $\mathrm{Pic}(\P^1)$ being isomorphic to  $\Z$. For a general curve Picard group is much more complicated and humongous (think of the Jacobian variety of a curve) so one doesn't get any further simplification. We now prove Theorem \ref{curve}, which is an extension of the previous proposition.

\begin{proof}[Proof of Theorem \ref{curve}]

		(3) $\implies$ (2): By Theorem \ref{conco}, $\@E$ is $\A^1$-concordant to $\@O_C^{n} \oplus \@L_1$, where $\@L_1 = \mathrm{det}(\@E)$ and $\@F$ is $\A^1$-concordant to $\@O_C^{n} \oplus \@L_2$, where $\@L_2 = \mathrm{det}(\@F)$. Hence $X$ and $\P_C(\@O_C^{n} \oplus \@L_1)$ are $\A^1$-$h$-cobordant. In the exact same manner, $Y$ and $\P_C(\@O_C^{n} \oplus \@L_2)$ are $\A^1$-$h$-cobordant. Suppose $\@L_1 \otimes \@L_2 ^{-1} = \@L^{\otimes n+1}$. That implies $\@L_1 = \@L^{\otimes n+1} \otimes \@L_2$ for some $\@L \in \Pic(C)$. Let $\@E' = (\@O_C ^n\oplus \@L_2) \otimes \@L $. Then $\mathrm{det}(\@E') = \@L_1$. Therefore $\P(\@E')$ is $\A^1$-$h$-cobordant to $\P(\@O_C^n \oplus \@L_1)$. Furthermore, $\P(\@E')$ is isomorphic (as a scheme) to $\P(\@O_C^n \oplus \@L_2)$ by the general fact that tensoring a vector bundle by a line bundle gives an isomorphism of projectivization of the two vector bundles. This proves $X$ and $Y$ are $\A^1$-$h$-cobordant.\\
		
		(2) $\implies$ (1): this is immediate from the definition of $\A^1$-$h$-cobordism.\\

(1) $\implies$ (3) :  $\A^1$-invariance of the Chow rings implies that it is enough to show that the Chow rings of $\P(\@O_C^{n}\oplus\@L_1)$ and $\P(\@O_C^{n}\oplus\@L_2)$ are not isomorphic if $\@L_1 \otimes \@L_2^{-1} \neq  \@L^{\otimes n+1}$ for any $\@L \in \Pic (C)$. The Chow ring of $C$ -- which is simply $\Z \oplus \Pic(C)$, with product of any two line bundles under the ring structure being zero -- is denoted $R$.  For simplicity of notation we will denote $\@O_C^n \oplus \@L_i$, $i=1,2$ by $\@E_i$.  Then by the projective bundle formula for the Chow rings \cite[Theorem 9.6]{EH}, the Chow ring of $\P(\@E_1)$ is $ R_1 := R[\zeta]/(\zeta^{n+1} + c_1(\@E_1)\zeta^{n})$. But $c_1(\@E_1) = c_1(\@L_1)$. In the ring $R_1$, $\zeta$ as well as any element $x \in \Pic(C)$ has grading $1$ with $xy=0$ for $x,y \in \Pic(C)$. Let's assume we have a graded ring isomorphism $\phi$ between $R_1$ and $R _2 := R[\sigma]/(\sigma^{n+1} + c_1(\@L_2)\sigma^{n})$. Then such an isomorphism has to respect the grading and hence $\phi(\zeta) = x + a\sigma$, where $x \in \Pic(C)$ and $a \in \Z$.  Similarly $\phi^{-1}(\sigma) = y + b\zeta$, where $b \in \Z$ and $y \in \Pic(C)$.  We first prove that $ a = \pm 1$. The condition $\phi^{-1}\circ \phi (\zeta) = \zeta$ implies that $x + ay + ab\zeta = \zeta$. Hence $ab =1$, so $a = \pm1$.\\By the graded ring structure of $R_1$, as discussed before, $x^i =0$ for any $i>1$. Moreover $\phi(\zeta^{n+1} + c_1(\@L_1)\zeta^{n})$ has to be divisible by $\sigma^{n+1} + c_1(\@L_2)\sigma^{n}$ in $R_2$. First assume $a=1$. Proof for the case $a=-1$ is similar. We expand $\phi(\zeta^{n+1} + c_1(\@L_1)\zeta^{n})$ as $\sigma^{n+1} + \sigma^n((n+1)x + c_1(\@L_1))$ and this expression is divisible by $\sigma^{n+1} + c_1(\@L_2)\sigma^{n}$. Comparing coefficients we conclude that  $c_1(\@L_1) - c_1(\@L_2) = (n+1)x$. This implies that  $\@L_1 \otimes \@L_2^{-1} = \@L^{\otimes n+1}$, where $c_1(\@L) =  x$.
	
\end{proof}
\noindent Now, we answer a question raised in \cite{AKW}, negatively. 
\begin{question}\cite[Question 6.9.1]{AKW}
If $X$ is any smooth projective variety that is $\A^1$-h-cobordant to a $\P^1$-bundle over $\P^2$, does $X$ have the structure of a $\P^1$-bundle over $\P^2$?
\end{question}
The authors further add the answer is possibly no and non-trivial rank three vector bundles over $\P^1$ deformable to the trivial one are the likely counterexamples. We now prove Theorem \ref{count} which shows that the example alluded to above is indeed a correct counterexample.
%\begin{proposition} \label{count} $ X:=\P_{\P^1}(\@E)$, where on $ \@E := \@O \oplus \@O(-1) \oplus \@O(1)$ on $\P^1$. Then $X$ doesn't have the structure of a $\P^1$-bundle over $\P^2$.
%\end{proposition}
\begin{proof} [Proof of Theorem \ref{count}]: By Theorem \ref{conco}, $ X:=\P(\@E) \xrightarrow{\pi} \P^1$ is $\A^1$-$h$-cobordant  to the trivial $\P^2$- bundle on $\P^1$, namely, $\P^1 \times \P^2$. However $X$ and $\P^1 \times \P^2$ are not isomorphic as schemes. By \cite[II, Exercise 7.9(b)]{hart} an isomorphism would imply that for some line bundle on $\P^1$, say $\@O(a)$ where $a \in \Z$,  $\@O(a) \otimes \@E \simeq \@O(a) \oplus \@O(a-1) \oplus \@O(a+1)\simeq \@O^{\oplus 3}$ is an isomorphism of vector bundles on $\P^1$, which can not happen. \\Now suppose $ X \xrightarrow{\theta} \P_{\P^2}(\@E') := Y \xrightarrow{\phi} \P^2$, with $\theta$ an isomorphism of schemes, for some rank $2$ vector bundle $\@E'$ on $\P^2$. We thus have the following diagram 

	\begin{center}
\begin{tikzcd}
	X \simeq Y \arrow[r, "\phi"] \arrow[d, "\pi"]
	& \P^2  \\ \P^1
	& 
\end{tikzcd}
	\end{center}
	
Without loss of generality we can assume (by twisting $\@E'$ with a suitable line bundle in $\Pic(\P^2) $ as $c_1(\@E' \otimes \@L) = c_1(\@E') + 2c_1(\@L)$), $c_1(\@E') \in \{0,1\}$. Since $Y$ is $\A^1$-weakly equivalent to the trivial bundle on $\P^2$, their Chow rings are isomorphic. By \cite[Lemma 4.5]{AKW}, we have $c_1(\@E')^2 - 4c_2(\@E') = 0$. So  $c_1(\@E') = 0 = c_2(\@E')$. It thus suffices to show that $\@E'$ splits as a direct sum of line bundles as this will prove that $\@E' \simeq \@O_{\P^2} \oplus \@O_{\P^2} $. By the assumption that $X \simeq Y$ this will imply that $X$ is isomorphic to $\P^2 \times \P^1$, which from the discussion in the first paragraph of this proof can not happen.

We will prove that $\phi$ has a section. Such a section will give the following short exact sequence.
\begin{equation}\label{finale}
0 \rightarrow \@L_1 \rightarrow \@E' \rightarrow \@L_2 \rightarrow 0
\end{equation}
As both Chern classes of $\@E'$ vanish, by the Whitney sum formula of Chern classes, both $\@L_1$ and $\@L_2$ will be trivial. Therefore such a short exact sequence has to be a split one. This will prove $\@E' \simeq \@O_{\P^2} \oplus \@O_{\P^2}$.\\
	Define $F \hookrightarrow Y$  to be $\theta \circ \pi^{-1}(z)$ for a point $ z \in \P^1$. By construction $F \simeq \P^2$. We claim $\phi$ maps $F$ isomorphically onto $\P^2$. First we claim that $\phi_{|F}$ is surjective. Suppose not, then $Z:=\phi(F)$ is either a point or an irreducible curve (not necessarily smooth) in $\P^2$.  Since $\phi: Y \to \P^2$ is a $\P^1$-bundle map, the fiber of $\phi$ over each point of $\P^2$  is $\P^1$. Therefore $Z$ can not be a point. So assume $Z$ is an irreducible curve in $\P^2$. Consider smooth points $z_1 \neq z_2 \in Z$. Then using flatness of $\phi$, we have $\phi^{-1}(z_i) \simeq \P^1 \subset F$ for $i=1,2$. However any two lines in $\P^2$ intersect, so $\phi^{-1}(z_1)$ and $\phi^{-1}(z_i)$ intersect in $F$(which is isomorphic to $\P^2$). This contradicts our assumption that $z_1 \neq z_2$.  This establishes the surjectivity of $\phi_{|F}$.  We also conclude $\phi|_{F}$ is a degree $d$ morphism to $\P^2$ with $d\geq 1$. \\
	We now show that $d=1$. This is achieved by comparing the graded ring isomorphism induced on the Chow rings of $X$ and $Y$. The Chow ring of $X$ is $ R _1 := \Z[x,y]/(x^2,y^3)$, where 
	\begin{enumerate}
	\item[(i)] $x$ is the divisor $\P^2$ as a fiber over a point of $\P^1$
	\item[(ii)] $y$ corresponds to a divisor $D'$, such that  the pushforward $\pi_{\ast}(\@O_X(D'))$ to $\P^1$ is the vector bundle $\@E$.
	\end{enumerate} 
Similarly the Chow ring of $Y$ is $ R_2:= \Z[s,t]/(s^2,t^3)$ where 
\begin{enumerate} 
\item[(i)] $t$ corresponds to fiber of $\P^1$ (as a degree $1$ curve in $\P^2$) via $\phi$ 
\item[(ii)]  $s$ corresponds to a divisor $D$, such that  the pushforward $\phi_{\ast}(\@O_Y(D))$ to $\P^2$ is the rank two vector bundle $\@E'$. 
\end{enumerate}
Let $\psi: R_1 \to R_2$ be an isomorphism of graded rings. Then $\psi(x) = as + bt$, where $a,b \in \Z$. We have the condition that $\psi(x^2) = \psi(x)^2 = a^2s^2 + 2abst + b^2t^2$ lies in the ideal generated by $s^2$ and $t^3$. This implies $b=0$. Morever $\psi^{-1}\circ \psi(x) = x$. Therefore $a = \pm 1$. In a similar fashion one proves that any isomorphism between $R_1$ and $R_2$ sends $y$ to $\pm t$. Therefore we conclude that the graded ring isomorphism between $R_1$ and $R_2$ is given by $x \mapsto \pm s$ and $y \mapsto \pm t$. This implies $s$ is equivalent (in the Chow ring) to the class of $F \simeq\P^2$. Grauert's theorem (\cite[III, Corollary 12.9]{hart}) implies that the intersection multiplicity of divisor $D$ corresponding to $s$ (see the description of $R_2$ above) with any fiber of the map $\phi$ is $1$. As $s$ and $F$ are equivalent in the Chow ring, the same holds for $F$. This can not happen unless $d=1$ because if not, one can consider a point $z'$ in $\P^2$ such that the set $\phi|_{\P^2}^{-1}(z') $ has more than one point. This will force $\phi^{-1}(z') = \P^1$ to intersect $\P^2$ in more than one point, meaning an intersection multiplicity of greater than $1$ which as we just proved can not happen. This proves $\phi|_{F}$ is an isomorphism onto $\P^2$ and hence establishes the existence of a section of $\phi$. Via the short exact sequence \ref{finale} this proves that $\@E'$ is a trivial rank 2 vector bundle on $\P^2$, and thereby finishes the proof.
	
\end{proof}
\noindent \textbf{Acknowledgements.} We thank the anonymous referees for their careful reading, comments and corrections. In particular Remark \ref{2.10} is the result of a question by one of the referees. The second-named author was supported by SFB 1085 Higher Invariants, University of Regensburg, NBHM fellowship of the Department of Atomic Energy, Govt.\ of India and Fulbright-Nehru Doctoral Research fellowship during this work and thanks the Department of Math, University of Southern California for the hospitality.

\end{document}